\documentclass[11pt]{amsart}
\usepackage{amssymb,mathrsfs,graphicx,enumerate}
\usepackage{colortbl}
\definecolor{black}{rgb}{0.0, 0.0, 0.0}
\definecolor{red}{rgb}{1.0, 0.5, 0.5}

\title[   ]{Non-contraction of intermediate admissible discontinuities for 3-D planar isentropic magnetohydrodynamics}

\author[Kang]{Moon-Jin Kang}
\address[Moon-Jin Kang]{\newline Laboratoire Jacques-Louis Lions, \newline University Pierre et Marie Curie, Paris, France}
\email{kang@ljll.math.upmc.fr}

\newtheorem{theorem}{Theorem}[section]

\newtheorem{remark}{Remark}[section]

\newcommand{\bbr}{\mathbb R}

\numberwithin{figure}{section}
\newcommand{\beq}{\begin{equation}}
\newcommand{\eeq}{\end{equation}}
\newcommand{\bsp}{\begin{split}}
\newcommand{\esp}{\end{split}}

\newcommand\adots{\mathinner{\mkern2mu\raise1pt\hbox{.}
\mkern3mu\raise4pt\hbox{.}\mkern1mu\raise7pt\hbox{.}}}

\def\charf {\mbox{{\text 1}\kern-.30em {\text l}}}

\begin{document}

\date{\today}
\bibliographystyle{plain}

\subjclass{    } \keywords{}

\thanks{\textbf{Acknowledgment.} This work was supported by the Foundation Sciences Math$\acute{\mbox{e}}$matiques de Paris as a postdoctoral fellowship, and by an AMS-Simons Travel Grant. The author thank Prof. Kevin Zumbrun for valuable comments on stability issues for MHD}

\begin{abstract}
We investigate non-contraction of large perturbations around intermediate entropic shock waves and contact discontinuities for the three-dimensional planar compressible isentropic magnetohydrodynamics (MHD). To do that, we take advantage of criteria developed by Kang and Vasseur in \cite{Kang-V-3}, and non-contraction property is measured by pseudo distance based on relative entropy.
\end{abstract}
\maketitle \centerline{\date}


\section{Introduction}
This article is devoted to the study of non-contraction property of certain intermediate entropic shock waves and contact discontinuities for the three-dimensional planar compressible inviscid isentropic MHD, which takes the form in Lagrangian coordinates: 
\begin{align}
\begin{aligned}\label{MHD}
\left\{ \begin{array}{ll}
       \partial_t v - \partial_x u_1 =0\\
          \partial_t (vB_2) - \beta \partial_x u_2 = 0\\
         \partial_t (vB_3) - \beta \partial_x u_3 = 0\\
          \partial_t u_1 + \partial_x (p+\frac{1}{2}(B_2^2+B_3^2)) = 0\\
          \partial_t u_2 - \beta \partial_x B_2 = 0\\
          \partial_t u_3 - \beta \partial_x B_3 = 0. \end{array} \right.
\end{aligned}
\end{align}
Here $v$ denotes specific volume, and $(u_1,u_2,u_3)$ and $(\beta,B_2,B_3)$ represent the three-dimensional fluid velocity and magnetic field, respectively. Those only depend on a single direction $e_1$ measured by $x$, and no dynamics with respect to other variables. Notice that $\beta$ is constant due to the divergence-free condition of magnetic field of full MHD.  As an ideal isentropic polytropic gas, the pressure $p$ is assumed to satisfy
\beq\label{pressure-MHD}
p(v)=v^{-\gamma},\quad \gamma >1.
\eeq
The system \eqref{MHD} has a convex entropy $\eta$ as
\[
\eta(U)=\int_{v}^{\infty} p(s) ds + \frac{1}{2}(u_1^2 + u_2^2+u_3^2) + \frac{1}{2v}(q_2^2+q_3^2)
\]
in terms of the conservative variables $U:=(v,q_2,q_3,u_1,u_2,u_3)$ where $q_i:=vB_i$ for $i=2,3$. In particular, $\eta$ is strictly convex because we consider non-vacuum states for MHD. \\
Using the entropy $\eta$, we define its relative entropy function by
\[
\eta(u|v)=\eta(u)-\eta(v) -\nabla\eta(v)\cdot(u-v).
\]
It is well-known that the relative entropy $\eta(\cdot |\cdot)$ is positive-definite, but looses the symmetry unless $\eta(u)=|u|^2$. Nevertheless the relative entropy is comparable to the square of $L^2$ distance for any bounded solutions. (See for example \cite{Kang-V-3, LV, Vasseur-2013}) Recently in \cite{Vasseur-2013}, Vasseur has shown contraction for large perturbations around extremal shocks ($1$-shock and $n$-shock) of the hyperbolic system of conservation laws satisfying physical conditions, which is satisfied by Euler systems of gas dynamics. In order to measure the distance between any bounded entropic solution and extremal shock, he used a spatially inhomogeneous pseudo-distance as follows:
for a given weight $a>0$, the spatially inhomogeneous pseudo-distance $d_a$ is defined by
\begin{align*}
\begin{aligned}
d_a(u(t,x),S(t,x))=\left\{ \begin{array}{ll}
         \eta (u(t,x)|u_l) & \mbox{if $  x < \sigma t$},\\
         a\eta (u(t,x)|u_r) & \mbox{if $ x>\sigma t$},\end{array} \right.
\end{aligned}
\end{align*}
where $S(t,x)$ denotes a given extremal shock $(u_l,u_r,\sigma)$, i.e.,
\[
S(t,x)=\left\{ \begin{array}{ll}
         u_l & \mbox{if $  x < \sigma t$},\\
         u_r & \mbox{if $ x> \sigma t$}.\end{array} \right.
\] 
Based on this pseudo-distance, it has been shown in \cite{Vasseur-2013} that there exists suitable weight $a>0$ such that the extremal shock is contractive up to suitable Lipschitz shift $\alpha(t)$ in the sense that for all bounded entropic solution $u\in BV_{loc}((0,\infty)\times\bbr)^n$,
\beq\label{sense}
 \int_{-\infty}^{\infty} d_a(u(t,x+\alpha(t)),S(t,x)) dx
\eeq
is non-increasing in time. On the other hand for intermediate admissible discontinuities, the authors in \cite{Kang-V-3} developed criteria to identify whether the intermediate entropic shocks and contact discontinuities are contractive or not in the pseudo-distance as above. Applying the criteria into the two-dimensional planar isentropic MHD, it turns out in \cite{Kang-V-3} that there is no weight for the contraction of certain intermediate shock waves.\\
In this article, we use the criteria developed in \cite{Kang-V-3} to show non-contraction of certain intermediate shocks and contact discontinuities. More precisely, we prove that there is no weight $a>0$ such that the intermediate discontinuities satisfy contraction in that sense \eqref{sense} with weight $a$.\\

Concerning studies on stability of shock waves to the viscous model of \eqref{MHD}, we refer to \cite{B-H-Zumbrun, B-L-Zumbrun, F-T, G-M-W-Zumbrun, M-Zumbrun}, in which it turns out that the viscous shock waves (including intermediate waves) are Evans stable under small perturbation. This implies Lopatinski stability because Lopatinski stability condition is necessary for stability of viscous profile (See \cite{Z-S}). Thus small $BV$-perturbations of the inviscid MHD entropic shock waves are stable thanks to Lopatinski stability (See \cite{Chern, Lewicka-2000, Lewicka}). Notice that this is not in contradiction with our result on non-contraction, because our framework is based on large perturbation around the entropic discontinuity, thus our results says that some large perturbation can increase as time goes on. In Remark \ref{remark-non}, we explain more precisely a meaning of non-contraction for entropic discontinuities in our framework.  

The rest of the paper is organized as follows. In Section 2, we present a criterion developed in \cite{Kang-V-3} and six characteristic fields of the $6\times 6$ system \eqref{MHD}, which are used to show non-contraction for intermediate entropic shock waves of \eqref{MHD} in Section 3, and for contact discontinuities in Section 4.

\section{Preliminaries}

\subsection{Criteria on non-contraction of intermediate entropic discontinuities}
In this section, we present a criterion in \cite{Kang-V-3} for non-contraction of admissible discontinuities of hyperbolic system of conservation laws:
\begin{align}
\begin{aligned} \label{system}
&\partial_t u + \partial_x f(u) = 0, \quad t>0,~x\in \bbr,\\
&u(0,x) = u_0(x).
\end{aligned}
\end{align}
For the system \eqref{system} satisfying Liu and Lax entropy conditions, the authors in \cite{Kang-V-3} have developed sufficient conditions to identify non-contraction of intermediate entropic discontinuities as follows. 
\begin{theorem}\label{thm-shock}
For a fixed $1<i<n$, let $(u_l,u_r, \sigma_{l,r})$ be a given $i$-th entropic discontinuity satisfying Liu and Lax entropy conditions. Assume that there are $1\le j<i<k\le n$ such that $j$- and $k$-characteristic fields are genuinely nonlinear. Then the following statements holds.\\
\begin{itemize}
\item (1) For $0<a<1$, we assume that there is a $C^1$ $j$-th rarefaction curve $R_{u_l}^j (s)$ with $j<i$ such that $R_{u_l}^j(0)=u_l$ and the backward curve $R_{u_l}^{j,-} (s)$ of $R_{u_l}^j (s)$, i.e., $\lambda_j (R_{u_l}^{j,-} (s)) < \lambda_j (u_l) $, intersects with the $(n-1)$-dimensional surface $\Sigma_a$.
Then, $(u_l,u_r, \sigma_{l,r})$ does not satisfy contraction in the sense \eqref{sense} with weight $a$.\\

\item (2) For $a>1$, we assume that there is a $C^1$ $k$-th rarefaction curve $R_{u_r}^k (s)$ with $k>i$ such that $R_{u_r}^k(0)=u_r$ and the forward curve $R_{u_r}^{k,+} (s)$ of $R_{u_r}^k (s)$, i.e., $\lambda_k (R_{u_r}^{k,+} (s)) > \lambda_k (u_r) $, intersects with the $(n-1)$-dimensional surface $\Sigma_a$.
Then, $(u_l,u_r, \sigma_{l,r})$ does not satisfy contraction in the sense \eqref{sense} with weight $a$.\\

\item  (3) For $a=1$, we assume that one of the assumptions of (1) and (2) is satisfied. Then, $(u_l,u_r, \sigma_{l,r})$ does not satisfy contraction in the sense \eqref{sense} with weight $a$.\\
\end{itemize}
\end{theorem}

\begin{remark}\label{remark-non}
From the meaning of contraction as mentioned in \eqref{sense}, we see the definition of non-contraction in our framework as follows. We say that an admissible discontinuity $(u_l,u_r, \sigma)$ does not satisfy contraction in the pseudo distance \eqref{sense} with weight $a$ if for any Lipschitz curve $\alpha(t)$ with $\alpha(0)=0$, there are some entropy solution $\bar{u}$ and small constant $T_0>0$ such that for all $0<t<T_0$,
\beq\label{bar-ineq}
\int_{-\infty}^{\infty} d_a(\bar{u}(t,x+\alpha(t)),S(t,x)) dx>\int_{-\infty}^{\infty} d_a(\bar{u}(0,x),S(0,x)) dx.
\eeq
In fact, the authors in \cite{Kang-V-3} constructed a specific (local) smooth solution $\bar{u}$ satisfying \eqref{bar-ineq}, which evolves from a smooth initial data 
\begin{align*}
\begin{aligned}
\bar{u}(0,x)=\left\{ \begin{array}{ll}
        \bar{u} & \mbox{if $x\in (-R,R)$ for some $R>0$},\\
         u_l & \mbox{if $x \in (-\infty, -2R)$},\\
         u_r & \mbox{if $x \in (2R, \infty)$},\end{array} \right.
\end{aligned}
\end{align*} 
where $\bar{u}$ is the point appeared in Theorem \ref{thm-shock}, which is the intersection point of suitable rarefaction wave and surface $\Sigma_a$, thus $\bar{u}$ depends on the weight $a$, thus on the pseudo distance.
\end{remark}

\subsection{Characteristic fields for the $6\times 6$ system \eqref{MHD}}
We here present six characteristic fields of the system \eqref{MHD}. For simplicity of computation, we use non-conservative variable $W:=(v,B_2,B_3,u_1,u_2,u_3)$ and rewrite \eqref{MHD} as a quasilinear form:
\[
 \partial_t W + A\partial_x W =0,
\]
where the $6\times 6$ matrix $A$ is given by
\[
A:= \left( \begin{matrix} 0&0&0&-1&0&0\\ 0&0&0&\frac{B_2}{v}&-\frac{\beta}{v}&0\\
0&0&0&\frac{B_3}{v}&0&-\frac{\beta}{v}\\-c^2& B_2&B_3& 0&0&0\\ 0&-\beta &0&0&0&0\\
0&0&-\beta &0&0&0
 \end{matrix} \right).
\]
where $c:=\sqrt{-p^{\prime}(v)}$ denotes the sound speed.\\
By a straightforward computation, we have the characteristic polynomial of $A$ as
\[
\Big(\lambda^2-\frac{\beta^2}{v}\Big)\Big(\lambda^4 -\Big(\frac{|B|^2 + \beta^2}{v} + c^2 \Big) \lambda^2 + \frac{\beta^2}{v}c^2\Big) =0,
\]
where $B:=(B_2, B_3)$. This equation has solutions $\lambda^2=\frac{\beta^2}{v}, \alpha_-, \alpha_+$, where $\alpha_{\pm}$ solve the quadratic equation 
\[
f(\Lambda):= \Lambda^2 -\Big(\frac{|B|^2 + \beta^2}{v} + c^2 \Big) \Lambda + \frac{\beta^2}{v}c^2=0,
\]
i.e.,
\[
\alpha_{\pm}:=\frac{1}{2} \Big[ \frac{|B|^2+\beta^2}{v} +c^2 \pm\sqrt{ \Big( \frac{|B|^2+\beta^2}{v}+c^2 \Big)^2-4\beta^2\frac{c^2}{v}} \Big].
\]\\
Then since 
\begin{align}
\begin{aligned}\label{f-relation}
f(\Lambda)&=(\Lambda-\frac{\beta^2}{v})(\Lambda-c^2)-\frac{|B|^2}{v}\Lambda\\
&\le (\Lambda-\frac{\beta^2}{v})(\Lambda-c^2),\quad\mbox{for}~ \Lambda>0,
\end{aligned}
\end{align}
we have
\[
\frac{\beta^2}{v}, c^2 \in[\alpha_{-}, \alpha_+].
\]
If we consider the case of $|B|\neq 0$,
then we have 
\beq\label{alpha-c}
\frac{\beta^2}{v}, c^2 \in(\alpha_{-}, \alpha_+).
\eeq
Here we assume that $\beta\neq 0$. Thus we have six eigenvalues
\[
\lambda_1=-\sqrt{\alpha_+},\quad \lambda_2=-\frac{\beta}{\sqrt{v}},\quad\lambda_3=-\sqrt{\alpha_-},\quad \lambda_4=\sqrt{\alpha_-},\quad \lambda_5=\frac{\beta}{\sqrt{v}},\quad \lambda_6=\sqrt{\alpha_+}.  
\]
By a straightforward computation, we have the corresponding eigenvectors
\begin{align*}
\begin{aligned}
r_1=\left( \begin{array}{c} \frac{v}{\alpha_+}(\alpha_+-\frac{\beta^2}{v}) \\ -B_2\\-B_3\\
\frac{v}{\sqrt{\alpha_+}}(\alpha_+-\frac{\beta^2}{v})\\-\frac{\beta B_2}{\sqrt{\alpha_+}}\\
-\frac{\beta B_3}{\sqrt{\alpha_+}} \end{array} \right),
r_2=\left( \begin{array}{c} 0 \\ \beta B_3\\-\beta B_2\\0\\\beta\sqrt{v}B_3\\-\beta\sqrt{v}B_2 \end{array} \right),
r_3=\left( \begin{array}{c} \frac{v}{\alpha_-}(\frac{\beta^2}{v}-\alpha_-) \\ B_2\\ B_3\\
\frac{v}{\sqrt{\alpha_-}}(\frac{\beta^2}{v}-\alpha_-)\\\frac{\beta B_2}{\sqrt{\alpha_-}}\\
\frac{\beta B_3}{\sqrt{\alpha_-}} \end{array} \right),\\
r_4=\left( \begin{array}{c} -\frac{v}{\alpha_-}(\frac{\beta^2}{v}-\alpha_-) \\ -B_2\\ -B_3\\
\frac{v}{\sqrt{\alpha_-}}(\frac{\beta^2}{v}-\alpha_-)\\\frac{\beta B_2}{\sqrt{\alpha_-}}\\
\frac{\beta B_3}{\sqrt{\alpha_-}} \end{array} \right),
r_5=\left( \begin{array}{c} 0 \\ \beta B_3\\-\beta B_2\\0\\-\beta\sqrt{v}B_3\\\beta\sqrt{v}B_2 \end{array} \right),
r_6=\left( \begin{array}{c} -\frac{v}{\alpha_+}(\alpha_+-\frac{\beta^2}{v}) \\ B_2\\ B_3\\
\frac{v}{\sqrt{\alpha_+}}(\alpha_+-\frac{\beta^2}{v})\\-\frac{\beta B_2}{\sqrt{\alpha_+}}\\
-\frac{\beta B_3}{\sqrt{\alpha_+}} \end{array} \right),
\end{aligned}
\end{align*}
Then we can check that for each $i=1,3,4,6$, $(\lambda_i,r_i)$ are genuinely nonlinear whereas $(\lambda_2,r_2)$ and $(\lambda_4,r_4)$ are linearly degenerate as follows. Indeed since
\begin{align*}
\begin{aligned}
d\lambda_1\cdot r_1 =\frac{1}{2\sqrt{\alpha_+}}\Big[\underbrace{-\partial_v\alpha_+  \frac{v}{\alpha_+}(\alpha_+-\frac{\beta^2}{v})}_{I_1} +\underbrace{\partial_{B_2}\alpha_+ B_2}_{I_2} +\underbrace{\partial_{B_3}\alpha_+ B_3}_{I_3} \Big],
\end{aligned}
\end{align*}
\begin{align*}
\begin{aligned}
I_1 &=\frac{1}{2}\Big[\frac{|B|^2+\beta^2}{v^2}+p^{\prime\prime}+\frac{(\frac{|B|^2+\beta^2}{v}-c^2)(\frac{|B|^2+\beta^2}{v^2}-p^{\prime\prime})+\frac{2|B|^2c^2}{v^2}+\frac{2|B|^2p^{\prime\prime}}{v}}{\sqrt{ \Big( \frac{|B|^2+\beta^2}{v}-c^2 \Big)^2+4\frac{|B|^2c^2}{v}}} \Big]
\frac{v}{\alpha_+}(\alpha_+-\frac{\beta^2}{v})\\
&>\frac{1}{2}\frac{\frac{2|B|^2c^2}{v^2}+\frac{2|B|^2p^{\prime\prime}}{v}}{\sqrt{ \Big( \frac{|B|^2+\beta^2}{v}-c^2 \Big)^2+4\frac{|B|^2c^2}{v}}}\frac{v}{\alpha_+}(\alpha_+-\frac{\beta^2}{v})\\
&>0\quad\mbox{by}~\eqref{alpha-c},
\end{aligned}
\end{align*}
and for each $i=2,3$,
\begin{align*}
\begin{aligned}
I_i &= \Big[\frac{1}{v}+\frac{(\frac{|B|^2+\beta^2}{v}-c^2)\frac{1}{v}+\frac{2c^2}{v}}{\sqrt{ \Big( \frac{|B|^2+\beta^2}{v}-c^2 \Big)^2+4\frac{|B|^2c^2}{v}}}  \Big]B_i^2\\
&>\frac{\frac{2c^2}{v} B_i^2}{\sqrt{ \Big( \frac{|B|^2+\beta^2}{v}-c^2 \Big)^2+4\frac{|B|^2c^2}{v}}}\ge0,
\end{aligned}
\end{align*}
we have $d\lambda_1\cdot r_1 > 0$. Using the similar argument, we have $d\lambda_i\cdot r_i > 0$ for $i=3,4,6$.
On the other hand, it is easy to get $d\lambda_i\cdot r_i=0$ for $i=2,5$.

\section{non-contraction of shock waves}
In this Section, we show that there is no weight $a>0$ such that certain entropic shock waves satisfy contraction in the sense \eqref{sense} with weight $a$.\\
Let $(U_l,U_r,\sigma_3)$ be any 3-shock waves satisfying the Rankine-Hugoniot condition:
\begin{align}
\begin{aligned}\label{MHD-RH}
-[u_1]&=\sigma_3[v],\\
-\beta [u_2]&= \sigma_3 [q_2],\\
-\beta [u_3]&= \sigma_3 [q_3],\\
[p]+\Big[\frac{q_2^2}{2v^2}+\frac{q_3^2}{2v^2}\Big]&=\sigma_3 [u_1],\\
-\beta \Big[\frac{q_2}{v}\Big]&= \sigma_3 [u_2],\\
-\beta \Big[\frac{q_3}{v}\Big]&= \sigma_3 [u_3],
\end{aligned}
\end{align}
where $[f]:=f_r-f_l$.\\
By Lax condition, $d\lambda_3\cdot r_3 >0$ implies that $-r_3(U_l)$ is a tangent vector at $U_l$ of the 3-shock curve $S_{U_l}^3$ issuing from $U_l$. Thus since $dv\cdot(-r_3)<0$ and $du_1\cdot(-r_3)<0$, we have
\beq\label{2-shock}
[v]<0\quad \mbox{and}\quad [u_1]<0. 
\eeq
This is true at least for weak shock. But, this relation can be justified for shock waves of arbitrary amplitude. We give this justification in the Appendix for the reader's convenience.\\
On the other hand, since it follows from \eqref{MHD-RH} that for each $i=2, 3$, $\beta^2 \Big[\frac{q_i}{v}\Big]= \sigma_3^2 [q_i]$, we have
\[
B_{i,r}=\frac{v_l-\beta^2\sigma_3^{-2}}{v_r-\beta^2\sigma_3^{-2}} B_{i,l},
\]
equivalently,
\[
B_{i,r}-B_{i,l} = \frac{[v]}{\beta^2\sigma_3^{-2}-v_r} B_{i,l}.
\]
Moreover, since \eqref{alpha-c} and Lax condition yield
\[
\frac{\beta^2}{v_r}>\alpha_-(U_r)=\lambda_{3}^2(U_r)>\sigma_3^2,
\]
and $[v]<0$, we get  
\begin{align}
\begin{aligned}\label{B-sign}
B_{i,r}-B_{i,l} =\left\{ \begin{array}{ll}
        <0\quad\mbox{if} ~B_{i,l}>0\\
        >0\quad\mbox{if} ~B_{i,l}<0\\
        =0\quad\mbox{if} ~B_{i,l}=0. \end{array} \right.
\end{aligned}
\end{align}
Notice that we do not see the sign of $v_l-\beta^2\sigma_3^{-2}$, thus of $B_{i,r}$, because we have not found the explicit formulation of the speed $\sigma_3$.   
Based on the observation above, we here consider the specific condition that the 3-shock wave satisfies one of the following cases :
\begin{align}
\begin{aligned}\label{2-B}
&\mbox{For each $i=2,3$, $(B_{i,l},B_{i,r})$ satisfies one of}\\
&\hspace{1cm} B_{i,l}>B_{i,r}\ge0~  \mbox{or}~B_{i,l}<B_{i,r}\le0~ \mbox{or}~B_{i,l}=B_{i,r}=0.
\end{aligned}
\end{align}
\begin{remark}
In the above assumption \eqref{2-B}, when $B_{2,l}=0$ and $B_{3,l}=0$ simultaneously, it follows from \eqref{f-relation} that 
\[
\{\frac{\beta^2}{v_l}, -p^{\prime}(v_l)\} = \{\alpha_{-} (U_l), \alpha_+(U_l)\}.
\]
If $\frac{\beta^2}{v_l} > -p^{\prime}(v_l)$, i.e., $\alpha_+(U_l)=\frac{\beta^2}{v_l}$, then $\lambda_1(U_l)=\lambda_2(U_l)=-\frac{\beta}{\sqrt{v_l}}$ But, the eigenspace corresponding to the eigenvalue $-\frac{\beta}{\sqrt{v_l}}$ is spanned by independent  two eigenvectors:
\[
(0,1,1,0,\sqrt{v_l},\sqrt{v_l})^T,\quad (0,1,-1,0,\sqrt{v_l},-\sqrt{v_l})^T.
\] 
This implies that $d\lambda_1(U)\cdot r_1(U)\neq 0$ except for $U=U_l$ as an umbilical point. Likewise, if $\frac{\beta^2}{v_l} < -p^{\prime}(v_l)$, i.e., $\alpha_-(U_l)=\frac{\beta^2}{v_l}$, then $d\lambda_3(U)\cdot r_3(U)\neq 0$ except for $U=U_l$. Thus for those singular cases, the 1- and 3-characteristic fields are still genuinely nonlinear. Similarly for the case of $B_{2,r}=0$ and $B_{3,r}=0$, the 4- and 6-characteristic fields are genuinely nonlinear as well.
\end{remark}

For a given 4-shock wave $(\tilde{U}_l, \tilde{U}_r, \sigma_4)$, using the same arguments as above, we have $[\tilde{v}]>0$, $[\tilde{u}]<0$, and
\[
\tilde{B}_{i,l}-\tilde{B}_{i,r} = \frac{[\tilde{v}]}{\tilde{v}_l-\beta^2\sigma_4^{-2}} \tilde{B}_{i,l}.
\]
Since \eqref{alpha-c} and Lax condition yield
\[
\frac{\beta^2}{v_l}>\alpha_-(U_l)=\lambda_{4}^2(U_l)>\sigma_4^2,
\]
we have
\begin{align*}
\begin{aligned}
\tilde{B}_{i,l}-\tilde{B}_{i,r} =\left\{ \begin{array}{ll}
        <0\quad\mbox{if} ~B_{i,r}>0\\
        >0\quad\mbox{if} ~B_{i,r}<0\\
        =0\quad\mbox{if} ~B_{i,r}=0. \end{array} \right.
\end{aligned}
\end{align*}
Thus we consider the analogous condition that the 4-shock wave satisfies one of the following cases :
\begin{align}
\begin{aligned}\label{3-B}
&\mbox{For each $i=2,3$, $(\tilde{B}_{i,l},\tilde{B}_{i,r})$ satisfies one of}\\
&\hspace{1cm} \tilde{B}_{i,r}>\tilde{B}_{i,l}\ge0~  \mbox{or}~\tilde{B}_{i,r}<\tilde{B}_{i,l}\le0~ \mbox{or}~\tilde{B}_{i,l}=\tilde{B}_{i,r}=0.
\end{aligned}
\end{align}

We are now ready to show that for any $a>0$, there is no $a$-contraction of such intermediate shocks as follows.\\\\
\begin{theorem}\label{thm-MHD}
Let $(U_l, U_r, \sigma_3)$ be a given 3-shock wave of the system \eqref{MHD}-\eqref{pressure-MHD}
satisfying \eqref{2-B}. Then there is no weight $a>0$ such that $(u_l,u_r)$ satisfies  contraction in the sense \eqref{sense} with weight $a$. Likewise, this result holds for a given 4-shock wave $(\tilde{U}_l, \tilde{U}_r, \sigma_4)$ satisfying \eqref{3-B}.
\end{theorem}
\begin{proof}
$\bullet$ {\bf Case of 3-shock wave:}
First of all, we show that for any $0<a<1$, the backward 1-rarefaction wave $R_{U_l}^{1,-}$ issuing from $U_l$ intersects with the hypersurface $\Sigma_a$ (with dimension 5), i.e.,
\[
\Sigma_a:=\{U~|~\eta(U|U_l)=a\eta(U|U_r) \}.
\]
Since $dv\cdot r_1=\frac{v}{\alpha_+}(\alpha_+-\frac{\beta^2}{v})>0$, $v$ is strictly monotone along the integral curve of the vector field $r_1$, which means that the 1-rarefaction wave can be parameterized by $v$. Moreover since $d\lambda_1\cdot r_1>0$, $-r_1$ is the tangent vector field of the backward 1-rarefaction wave $R_{U_l}^{1,-}$, which implies that $v$ decreases along $R_{U_l}^{1,-}$. That is, $v_+\le v_l$ for all parameters $v_+$ of $R_{U_l}^{1,-}$. Notice that $R_{U_l}^{1,-}$ is well-defined for all $v_+\in(0,v_l]$, because $-r_1(W)$ is smooth for all $W\in (0,\infty)\times \bbr^5$. \\
In order to claim that $R_{U_l}^{1,-}$ intersects with $\Sigma_a$ for any $a<1$, we use the fact that
\begin{align}
\begin{aligned}\label{a<1}
&\mbox{for}~a<1,\quad\eta(U|U_l) \le a\eta(U|U_r) \quad\mbox{is equivalent to}\\
&\eta(U) \le \frac{1}{1-a} (\eta(U_l)-a\eta(U_r) -\nabla \eta(U_l)\cdot U_l + a\nabla \eta(U_r)\cdot U_r +(\nabla \eta(U_l)-a\nabla \eta(U_r))\cdot U),
\end{aligned} 
\end{align}
which is rewritten as
\beq\label{quad-mhd}
\int_{v}^{\infty} p(s) ds + \frac{1}{2}(u_1^2 + u_2^2+u_3^2) + \frac{1}{2v}(q_2^2+q_3^2)
\le c_1+c_2 (v+q_2+q_3+u_1+u_2+u_3),  
\eeq
for some constants $c_1, c_2$. This implies that 
\begin{align*}
\begin{aligned}
&\eta(U|U_l) \le a\eta(U|U_r)\\
&~ \Longleftrightarrow v>c_*~\mbox{and}~ |q_2|+|q_3|+ |u_1|+ |u_2|+|u_3|\le c^*\quad\mbox{for some constants}~c_*,~c^*>0, 
\end{aligned} 
\end{align*}
since $\int_{0}^{\infty}p(s) ds =+\infty$, and all positive terms on $u_i$ and $q_i$ are quadratic in the left-hand side, whereas linear in the right-hand side of \eqref{quad-mhd}. Therefore there exists $0<v_* \ll c_*$ such that
\[
\eta(R_{U_l}^{1,-}(v_*)|U_l) > a\eta(R_{U_l}^{1,-}(v_*)|U_r),
\]
which implies that $R_{U_l}^{1,-}$ intersects with $\Sigma_a$ for $a<1$, because $R_{U_l}^{1,-}$ is a continuous curve issuing from $U_l\in \{U~|~\eta(U|U_l) < a\eta(U|U_r)\}$.\\

On the other hand, we claim that the forward 6-rarefaction wave $R_{U_r}^{6,+}$ issuing from $U_r$ intersects with the surface $\Sigma_a$ for any $a\ge 1$.\\
Since $d\lambda_6\cdot r_6>0$ and $dv\cdot r_6=-\frac{v}{\alpha_+}(\alpha_+-\frac{\beta^2}{v})<0$,
\beq\label{-r}
r_6~\mbox{ is the tangent vector of the forward 6-rarefaction wave}~ R_{U_r}^{6,+},
\eeq
and the parameter $v_+$ decreases along $R_{U_r}^{6,+}$. Moreover $R_{U_r}^{6,+}$ is well-defined for all $v_+\in(0,v_r]$.\\
For any fixed $a\ge 1$, we consider a continuous functional $F_a$ defined by 
\[
F_a(U):=\eta(U|U_l)-a\eta(U|U_r).
\]
We first show that the functional $F_1$ (when $a=1$) satisfies
\beq\label{positive-F}
F_1(R_{U_r}^{6,+}(v_*)) < 0 \quad\mbox{for some}~v_*\in(0,v_r].
\eeq
Since 
\[
\nabla\eta(U)=\Big(-p-\frac{q_2^2+q_3^2}{2v^2}, \frac{q_2}{v}, \frac{q_3}{v}, u_1, u_2, u_3 \Big)^T,
\]
we use \eqref{-r} to compute
\begin{align}
\begin{aligned}\label{comp-1}
\frac{dF_1(R_{U_r}^{6,+}(v_+))}{dv_+}&=(\nabla \eta(U_r)-\nabla \eta(U_l))\cdot \frac{d R_{U_r}^{6,+}(v_+)}{dv_+}\\
&=\Big([p]+\Big[\frac{q_2^2+q_3^2}{2v^2}\Big]\Big) \frac{v_+}{\alpha_+}(\alpha_+ -\frac{\beta}{v_+})+[u_1]\frac{v_+}{\sqrt{\alpha_+}}(\alpha_+ -\frac{\beta}{v_+}) \\
&\quad +\sum_{i=2}^3\Big(\Big[\frac{q_i}{v}\Big] \frac{q_{i+}}{v_+} -[u_i] \frac{\beta q_{i+}}{v_+\sqrt{\alpha_+}} \Big).
\end{aligned}
\end{align}
Using \eqref{MHD-RH}, we have
\begin{align*}
\begin{aligned}
\frac{dF_1(R_{U_r}^{6,+}(v_+))}{dv_+}&= [u_1](\sigma_3 +\sqrt{\alpha_+})\frac{v_+}{\alpha_+}(\alpha_+ -\frac{\beta}{v_+})+\sum_{i=2}^3 \Big[\frac{q_i}{v}\Big] \frac{q_{i+}}{v_+} \Big( 1 + \frac{\beta^2 }{\sigma_3 \sqrt{\alpha_+}} \Big)\\
&=\underbrace{[u_1]\Big(v_+\sqrt{\alpha_+} -\frac{\beta}{\sqrt{\alpha_+}} +\sigma_3( v_+ -\frac{\beta}{\alpha_+})\Big)}_{I_1}  +\underbrace{\sum_{i=2}^3 \Big[\frac{q_i}{v}\Big] \frac{q_{i+}}{v_+} \Big( 1 + \frac{\beta^2 }{\sigma_3 \sqrt{\alpha_+}} \Big)}_{I_2}. 
\end{aligned}
\end{align*}
Since \eqref{pressure-MHD} and \eqref{alpha-c} yields
\beq\label{alpha-to}
v_+ \sqrt{\alpha_+} > v_+\sqrt{-p^{\prime}(v_+)} =\sqrt{\gamma v_+^{-\gamma+1}} \to \infty\quad\mbox{as}~v_+\to0+,
\eeq
it follows from \eqref{2-shock} that $I_1 \to -\infty$ as $v_+\to0+$. \\
To control $I_2$, we consider one of conditions in \eqref{2-B}. If $B_{2,l}>B_{2,r}\ge0$, i.e., $\Big[\frac{q_2}{v}\Big]<0$, we have $q_{2+}\ge0$ along $R_{U_r}^{6,+}(v_+)$ because of $q_{2,r}=v_rB_{2,r}\ge0$ and \eqref{-r} with
\begin{align*}
\begin{aligned}
dB_2\cdot r_6=B_2=\left\{ \begin{array}{ll}
       >0\quad\mbox{if} ~B_2>0\\
        <0\quad\mbox{if} ~B_2<0\\
        =0\quad\mbox{if} ~B_2=0. \end{array} \right.
\end{aligned}
\end{align*}
Moreover, since $\alpha_+\to+\infty$ as $v_+\to0+$ by \eqref{alpha-to}, we have
\[
 \Big[\frac{q_2}{v}\Big] \frac{q_{2+}}{v_+} \Big( 1 + \frac{\beta^2}{\sigma_3 \sqrt{\alpha_+}} \Big) \le 0\quad\mbox{for}~ v_+\ll 1.
\]
This is also true in the case of $B_{2,l}<B_{2,r}\le 0$, i.e., $\Big[\frac{q_2}{v}\Big]>0$, because of $q_{2+}\le0$ in that case by the same arguments as above.\\
Likewise, we have
\[
 \Big[\frac{q_3}{v}\Big] \frac{q_{3+}}{v_+} \Big( 1 + \frac{\beta^2}{\sigma_3 \sqrt{\alpha_+}} \Big) \le 0\quad\mbox{for}~ v_+\ll 1.
\] 
Thus the condition \eqref{2-B} yields
\[
I_2 \le 0 \quad\mbox{for}~ v_+\ll 1,
\]
which yields
\[
\frac{dF_1(R_{U_r}^{6,+}(v_+))}{dv_+} \to -\infty \quad\mbox{as}~v_+\to0+,
\]
which implies \eqref{positive-F}.\\
Therefore we conclude that $R_{U_r}^{6,+}$ intersects with $\Sigma_a$ for any $a\ge 1$, because $F_a(U_r)>0$ and $F_a(R_{U_r}^{6,+}(v_*))<F_1(R_{U_r}^{6,+}(v_*))<0$ for all $a\ge1$.\\

Hence for all $a>0$, the 3-shock wave $(U_l,U_r,\sigma_2)$ does not satisfies contraction in the sense \eqref{sense} with weight $a$ thanks to Theorem \ref{thm-shock}.\\

$\bullet$ {\bf Case of 4-shock wave :} Following the same arguments as above in a symmetric way, we have the non-contraction for 4-shock wave $(\tilde{U}_l, \tilde{U}_r, \sigma_4)$ satisfying \eqref{3-B}. More precisely, we can show that the backward 1-rarefaction wave $R_{\tilde{U}_l}^{1,-}$ intersects with 
\[
\tilde{\Sigma}_a:=\{U~|~\eta(U|\tilde{U}_l)=a\eta(U|\tilde{U}_r) \}\quad \mbox{for any}~0<a\le 1,
\]
and the forward 6-rarefaction wave $R_{\tilde{U}_r}^{6,+}$ intersects with $\tilde{\Sigma}_a$ for any $a>1$. We omit the details.
\end{proof}

\begin{remark}
In the proof of Theorem \ref{thm-MHD}, we used the condition \eqref{2-B} only to ensure the intersection of $R_{U_r}^{6,+}$ with the hyperplane $\Sigma_1$, and similarly the condition \eqref{3-B} only for the intersection of $R_{\tilde{U}_l}^{1,-}$ with $\tilde{\Sigma}_1$. In other words, $R_{{U}_l}^{1,-}$ (resp. $R_{\tilde{U}_l}^{1,-}$) intersects with $\Sigma_a$ (resp. $\tilde{\Sigma}_a$) for $a<1$, and $R_{{U}_r}^{6,+}$ (resp. $R_{\tilde{U}_r}^{6,+}$) intersects with $\Sigma_a$ (resp. $\tilde{\Sigma}_a$) for $a>1$ without the condition \eqref{2-B} (resp. \eqref{3-B}). 
\end{remark}

\section{non-contraction of contact discontinuities}
We here show that there is no weight $a>0$ such that certain contact discontinuities satisfy contraction in the sense \eqref{sense} with weight $a$.\\
Let $(U_l, U_r)$ be a given 2-contact discontinuity (or 5-contact discontinuity) of the system \eqref{MHD}-\eqref{pressure-MHD}. Since $i$-contact discontinuity $(U_l,U_r)$ is a integral curve of the vector field $r_i$ for each $i=2,5$, we have
\beq\label{contact-2}
[v]=0\quad\mbox{and}\quad [u_1]=0,
\eeq
which implies that $[B_i]\neq0$ for some $i=2,3$, otherwise $U_l=U_r$.\\ 
Contrary to the case of shock waves, there is no sign of $[B_2]$ and $[B_3]$ because of $\sigma_2^2\equiv \frac{\beta^2}{v_l}=\frac{\beta^2}{v_r}$, which is due to the degeneracy of contact discontinuity.\\
We here present non-contraction of contact discontinuities satisfying either $(A)$ or $(B)$: 
\begin{align}
\begin{aligned}\label{contact-B}
&(A)~:~\mbox{For each $i=2,3$, $(B_{i,l},B_{i,r})$ satisfies one of}\\
&\hspace{1cm} B_{i,r}>B_{i,l}>0~  \mbox{or}~B_{i,r}<B_{i,l}<0,\\
&(B)~:~\mbox{For each $i=2,3$, $(B_{i,l},B_{i,r})$ satisfies one of}\\
&\hspace{1cm} B_{i,l}>B_{i,r}>0~  \mbox{or}~B_{i,l}<B_{i,r}<0.
\end{aligned}
\end{align}

\begin{theorem}\label{thm-contact}
Let $(U_l, U_r)$ be a given 2-contact discontinuity (or 5-contact discontinuity) of the system \eqref{MHD}-\eqref{pressure-MHD}. Then there is no weight $a>0$ such that $(u_l,u_r)$ satisfies contraction in the sense \eqref{sense} with weight $a$.
\end{theorem}
\begin{proof} 
We follow the same arguments as the proof of Theorem \ref{thm-MHD}. First of all,  we can see that the backward 1-rarefaction wave $R_{U_l}^{1,-}$ issuing from $U_l$ intersects with $\Sigma_a$ for $a<1$, and the forward 6-rarefaction wave $R_{U_r}^{6,+}$ issuing from $U_r$ intersects with $\Sigma_a$ for $a>1$.\\
In order to show that one of $R_{U_l}^{1,-}$ and $R_{U_r}^{6,+}$ intersects with the $\Sigma_1$, we consider a functional
\[
F(U):=\eta(U|U_l)-\eta(U|U_r).
\]
$\bullet$ {\bf Case (A)} (For each $i=2,3$, $(B_{i,l},B_{i,r})$ satisfies one of $B_{i,r}>B_{i,l}>0$ or $B_{i,r}<B_{i,l}<0$)\\
In this case, we claim that the wave $R_{U_l}^{1,-}$ parametrized as $v_+$  intersects with $\Sigma_1$. First of all, using the same computation as \eqref{comp-1} with\eqref{contact-2}, we have
\[
\frac{dF(R_{U_l}^{1,-}(v_+))}{dv_+}= \sum_{i=2}^3 \Big[\frac{q_i}{v}\Big] \frac{q_{i+}}{v_+} \Big( 1 - \frac{\beta^2 }{\sigma_2 \sqrt{\alpha_+}} \Big).
\]
If $B_{2,r}>B_{2,l}>0$, i.e., $\Big[\frac{q_2}{v}\Big]>0$, since 
$-r_1$ is the tangent vector field of the backward 1-rarefaction wave $R_{U_l}^{1,-}$, and $dB_2\cdot (-r_1)=B_2$,
we have $\frac{q_{2+}}{v_+}\to \infty$ as $v_+\to 0+$, thus
\[
 \Big[\frac{q_2}{v}\Big] \frac{q_{2+}}{v_+} \Big( 1 + \frac{\beta^2}{\sigma_2 \sqrt{\alpha_+}} \Big) \to \infty\quad\mbox{for}~ v_+\to 0+,
\]
where we have used the fact that $\alpha_+\to \infty$ as $v_+\to 0+$ by \eqref{alpha-to}.\\
This is also true in the case of $B_{2,r}<B_{2,l}< 0$, because of $\Big[\frac{q_2}{v}\Big]<0$ and $\frac{q_{2+}}{v_+}\to -\infty$ as $v_+\to 0+$. 
Likewise, we have
\[
 \Big[\frac{q_3}{v}\Big] \frac{q_{3+}}{v_+} \Big( 1 + \frac{\beta^2}{\sigma_2 \sqrt{\alpha_+}} \Big)\to \infty\quad\mbox{for}~ v_+\to 0+.
  \] 
Thus we have
\[
\frac{dF(R_{U_l}^{1,-}(v_+))}{dv_+} \to \infty \quad\mbox{as}~v_+\to0+,
\]
which implies that $R_{U_l}^{1,-}$ intersects with $\Sigma_1$, because $F(U_l)<0$ and $F(R_{U_l}^{1,-}(v_*))>0$ for some $v_*\ll 1$.\\
$\bullet$ {\bf Case (B)} (For each $i=2,3$, $(B_{i,l},B_{i,r})$ satisfies one of $B_{i,l}>B_{i,r}>0$ or $B_{i,l}<B_{i,r}<0$)\\
Using the same arguments as previous case, we see that $R_{U_r}^{6,+}$ intersects with $\Sigma_1$ under those constraints.
\end{proof}

\section{appendix}
We here present that the relations \eqref{2-shock} holds true for 3-shock waves of arbitrary amplitude. Using \eqref{MHD-RH} and the entropy inequality with the entropy flux 
\[
G:=\Big(p+\frac{q_2^2+q_3^2}{2v^2}\Big)u_1-\frac{\beta}{v}(q_2u_2+q_3u_3),
\]
we have
\begin{align*}
\begin{aligned}
0&\ge [G]-\sigma_3[\eta]\\
&=[u_1]p_l + u_{1,r}[p] +\sum_{i=2}^3\Big(\frac{1}{2}[u_1]\frac{q_{i,l}^2}{v_l^2}+u_{1,r}\Big[\frac{q_{i}^2}{2v^2}\Big]\Big)-\sigma_3 [u_1]\frac{u_{1,r}+u_{1,l}}{2}\\
&\quad-\sum_{i=2}^3\Big(\beta \Big[\frac{q_iu_i}{v}\Big]+\frac{\sigma_3}{2} [u_i](u_{i,r}+u_{i,l})\Big)-\frac{\sigma_3}{2}\sum_{i=2}^3\Big[\frac{q_i^2}{v}\Big]-\sigma_3\int_{v_r}^{v_l} p(s)ds\\
&=[u_1]p_l +\frac{1}{2}[u_1]\sum_{i=2}^3\frac{q_{i,l}^2}{v_l^2}+ u_{1,r}[p] +u_{1,r}\sum_{i=2}^3\Big[\frac{q_i^2}{2v^2}\Big]- \Big([p]+\sum_{i=2}^3\Big[\frac{q_i^2}{2v^2}\Big]\Big)\frac{u_{1,r}+u_{1,l}}{2}\\
&\quad-\beta\sum_{i=2}^3\Big( [u_i]\frac{q_{i,l}}{v_l}+ \Big[\frac{q_i}{v}\Big]u_{i,r}\Big)+\frac{\beta}{2} \sum_{i=2}^3\Big[\frac{q_i}{v}\Big](u_{i,r}+u_{i,l})\\
&\quad +\frac{\sigma_3}{2}\sum_{i=2}^3\Big(\frac{q_{i,l}^2}{v_{r}v_{l}}[v]-\frac{[q]}{v_{r}}(q_{i,r}+q_{i,l})\Big)-\sigma_3\int_{v_r}^{v_l} p(s)ds\\
&=[u_1]p_l +\frac{1}{2}[u_1]\sum_{i=2}^3\frac{q_{i,l}^2}{v_l^2}+ \frac{[u_1]}{2}[p] +\frac{[u_1]}{2} \sum_{i=2}^3\Big[\frac{q_i^2}{2v^2}\Big] \\
&\quad-\sum_{i=2}^3\beta [u_i]\Big(\frac{q_{i,l}}{v_l}+\frac{1}{2} \Big[\frac{q_i}{v}\Big]\Big)+\frac{\sigma_3}{2}\sum_{i=2}^3\Big(\frac{q_{i,l}^2}{v_rv_l}[v]-\frac{[q_i]}{v_r}(q_{i,r}+q_{i,l})\Big)-\sigma_3\int_{v_r}^{v_l} p(s)ds,\\
\end{aligned}
\end{align*}
using \eqref{MHD-RH} again, 
\begin{align*}
\begin{aligned}
&= -\sigma_3[v] p_l +\frac{\sigma_3}{2}[v]\sum_{i=2}^3\frac{q_{i,l}^2}{v_l^2}- \frac{\sigma_3}{2}[v][p] +\frac{\sigma_3}{4}[v]\sum_{i=2}^3\Big[\frac{q_i^2}{v^2}\Big] \\
&\quad+\sum_{i=2}^3 \sigma_3[q_i] \Big(\frac{q_{i,l}}{v_l}+ \frac{1}{2}\Big[\frac{q_i}{v}\Big]\Big)+\frac{\sigma_3}{2}\sum_{i=2}^3\Big(\frac{q_{i,l}^2}{v_rv_l}[v]-\frac{[q_i]}{v_r}(q_{i,r}+q_{i,l})\Big)-\sigma_3\int_{v_r}^{v_l} p(s)ds\\
&=-\sigma_3\Big(\Big(\frac{[v]}{2}(p_r+p_l)- \int_{v_l}^{v_r} p(s)ds \Big) + \sum_{i=2}^3\Big(\frac{[v]}{2}\Big(\frac{q_{i,r}^2}{v_r^2}+\frac{q_{i,l}^2}{v_l^2} -\frac{q_{i,l}^2}{v_rv_l}\Big)-\frac{[q_i]}{2}\Big(\frac{q_{i,l}}{v_l} +\frac{q_{i,r}}{v_r}-\frac{q_{i,r}+q_{i,l}}{v_r}\Big) \Big)\Big)\\
&=-\sigma_3[v]\Big(\int_0^1\Big(sp_r +(1-s)p_l -p(sv_r+(1-s)v_l)\Big)ds +\sum_{i=2}^3\frac{1}{4}\Big(\frac{q_{i,r}}{v_r} -\frac{q_{i,l}}{v_l}\Big)^2 \Big).
\end{aligned}
\end{align*}
Finally, using the convexity of $p$, we have
\[
\sigma_3[v] \ge 0.
\]
Thus by Lax condition $\sigma_3<\lambda_3(U_l)<0$, we have $v_r<v_l$, which yields $u_r<u_l$ by \eqref{MHD-RH}.  
\bibliography{Kang-Vasseur2015}
\end{document}